\documentclass[a4paper,11pt]{article}
\usepackage[left=20mm,right=20mm,textheight=700pt]{geometry} 

\usepackage{type1cm}
\usepackage{amsmath, amsthm, amscd, amsfonts, amssymb, graphicx, color}
\usepackage[abbrev]{amsrefs} 
\usepackage{hyphenat}
\usepackage[english]{babel}
\usepackage{xcolor}
\usepackage{tikz}
\usetikzlibrary{arrows.meta}
\usetikzlibrary{positioning}
\usepackage[mathscr]{euscript} 

\usepackage{enumitem}
\usepackage{subcaption}
\usepackage{accents} 

\makeatletter
\DeclareRobustCommand{\cev}[1]{%
  \mathpalette\do@cev{#1}%
}
\newcommand{\do@cev}[2]{%
  \fix@cev{#1}{+}%
  \reflectbox{$\m@th#1\vec{\reflectbox{$\fix@cev{#1}{-}\m@th#1#2\fix@cev{#1}{+}$}}$}%
  \fix@cev{#1}{-}%
}
\newcommand{\fix@cev}[2]{%
  \ifx#1\displaystyle
    \mkern#23mu
  \else
    \ifx#1\textstyle
      \mkern#23mu
    \else
      \ifx#1\scriptstyle
        \mkern#22mu
      \else
        \mkern#22mu
      \fi
    \fi
  \fi
}
\makeatother

\PassOptionsToPackage{hyphens}{url}
\usepackage[                
   pdfstartview=FitV,       
   hidelinks,               
   hyperfootnotes=true,     
   hyperindex=true,         
   bookmarks=true,
   ]{hyperref}              

\usepackage{cleveref} 

\tikzstyle{mynodebase} = [draw,thick,circle,fill=blue!20,scale=1.75,inner sep=1.5,font={\scriptsize},minimum size=0.5cm]
\tikzstyle{mynode} = [mynodebase,fill=blue!20]
\tikzstyle{myedge} = [ultra thick]
\tikzstyle{myarc} = [ultra thick,->]
\tikzstyle{mybiarc} = [ultra thick,<->]

\newcommand{\diag}{\mathop{\mathrm{diag}}}

\newcommand{\RR}{\mathbb{R}}                                                    
\newcommand{\CC}{\mathbb{C}}                                                    

\newcommand{\gr}[1]{{\mathscr{#1}}}

\DeclareMathAlphabet{\mathpzc}{OT1}{pzc}{m}{it} 
\DeclareFontFamily{OT1}{pzc}{}
\DeclareFontShape{OT1}{pzc}{m}{it}{<-> s * [1.000] pzcmi7t}{}
\DeclareMathAlphabet{\mathpzc}{OT1}{pzc}{m}{it}
\newcommand{\hal}{{\mathpzc{h}_{\alpha}}} 
\newcommand{\hga}{{\mathpzc{h}_{\gamma}}} 

\newcommand{\OR}{{\Delta}} 
\newcommand{\ORd}{{\delta}} 
\newcommand{\HG}{{H_\gamma}} 
\newcommand{\HA}{{H_\alpha}} 
\newcommand{\AL}{{L_\gamma}} 

\newcommand{\Bp}{{R}}

\allowdisplaybreaks[1] 

\newtheorem{theorem}{Theorem}
\newtheorem{corollary}[theorem]{Corollary}

\newtheorem{remark}[theorem]{Remark}%

\newtheorem{definition}[theorem]{Definition}%
\newtheorem{example}[theorem]{Example}%
\newtheorem{construction}[theorem]{Construction}%

\begin{document}


\title{Incidence matrices and line graphs of mixed graphs}

\author{Mohammad Abudayah\footnote{School of Basic Sciences and Humanities, German Jordanian University, Amman, Jordan, mohammad.abudayah@gju.edu.jo}, 
Omar Alomari\footnote{School of Basic Sciences and Humanities, German Jordanian University, Amman, Jordan, omar.alomari@gju.edu.jo},
Torsten Sander\footnote{Fakult\"at f\"ur Informatik, Ostfalia Hochschule f\"ur angewandte Wissenschaften, Wolfenb\"uttel, Germany, t.sander@ostfalia.de}}

\maketitle

\begin{abstract}
In the theory of line graphs of undirected graphs there exists an important theorem linking
the incidence matrix of the root graph to the adjacency matrix of its line graph. For directed or mixed graphs, however, 
the exists no analogous result. The goal of this article is to present aligned definitions of the adjacency matrix,
the incidence matrix and line graph of a mixed graph such that the mentioned theorem is valid for mixed graphs.

\medskip


\end{abstract}

\section{Introduction}\label{intro}  

Line graphs have been an invaluable concept in graph theory for a long time.
The line graph $L(\gr{G})$ of an undirected graph $\gr{G}=(V,E)$ has vertex set $E$. Two vertices $e_1,e_2$ of $L(\gr{G})$ are adjacent if and only
if the edges $e_1,e_2$ are adjacent in $\gr{G}$. Hemminger and Beineke once called this
``probably the most interesting of all graph transformations'' \cite{Hemminger}.
As a direct consequence of the line graph definition, we have (cf.\ Lemma 3.6 in \cite{Biggs})
\begin{equation}\label{eq:btb}
  B^\ast B= A(L(\gr{G})) + 2I,
\end{equation}
where $B$ denotes the incidence matrix of $\gr{G}$, $B^\ast$ its conjugated transpose, $A(L(\gr{G}))$ denotes the adjacency matrix of $L(\gr{G})$ and
$I$ is the identity matrix.
Recall that the adjacency matrix of an undirected (resp.\ directed) graph on $n$ vertices is the $0$-$1$-matrix
such that the entry at position $(i,j)$ is $1$ if there is an edge between vertex no.\ $i$ and vertex no.\ $j$
(resp.\ from vertex no.\ $i$ to vertex no.\ $j$), and $0$ otherwise. 
The incidence matrix $B$ of an undirected graph on $n$ vertices is the $0$-$1$-matrix
such that the entry at position $(i,j)$ is $1$ if vertex no.\ $i$ and edge no.\ $j$
are incident, and $0$ otherwise.
Thus, there is a natural algebraic link between the matrix $B$ that captures
the incidence relation of $\gr{G}$ and the matrix $A(L(\gr{G}))$ that captures the adjacency relation of $L(\gr{G})$.

For directed graphs, the situation appears less satisfying. The line graph $L(\gr{D})$ of a directed digraph $\gr{D}=(V,E)$ (also called
the line digraph $L(\gr{D})$) has vertex set $E$. There is an arc from vertex $e_1$ to $e_2$ in $L(\gr{D})$ if and only
if the head vertex (i.e.\ terminal vertex) of the arc $e_1$ is the same as the tail vertex (i.e.\ initial vertex) of the arc $e_2$ in $\gr{D}$.
Clearly, the matrix $A(L(\gr{D}))$ is in general not symmetric. Hence, regardless of how we define
the incidence matrix of $\gr{D}$, the left-hand side of \eqref{eq:btb} is a symmetric matrix whereas the right-hand side is not.
To make things worse, when defining incidence matrices of directed graphs, one traditionally only refers to
oriented graphs, i.e.\ double arcs between vertices are forbidden (cf.\ \cite{Bapat}, \cite{Biggs}, \cite{Cvet}).
Moreover, for mixed graphs, no meaningful definition seems to exist at all. A mixed graph is a graph that has been derived from
an undirected graph by orienting some of its edges into arcs, while the unmodified edges remain as digons.  

In this paper, we will focus on mixed graphs instead of directed graphs since we want the number of edges to remain the same whenever
we transition to the underlying undirected simple graph. Our goal is to develop 
consistently aligned definitions of the adjacency matrix, the incidence matrix and the line graph of a mixed graph such that
equation \eqref{eq:btb} holds, moreover the well-known equation (cf.\ \cite{Cvet})
\begin{equation}\label{eq:bbt}
  B B^\ast = A(\gr{G}) + D
\end{equation}
that links the incidence matrix of a graph to its adjacency matrix. Here, $D$ denotes the diagonal matrix with the
respective degrees of all vertices of $\gr{G}$. It will turn out that, once suitable definitions of adjacency and incidence
matrix have been chosen, a matching notion of line graph arises in a perfectly natural way.

The rest of the paper is organized as follows. In section \ref{sec:line} the desired notion of line graph will be developed.
Section \ref{sec:root} will be devoted to answering the question how many mixed orientations of a graph may result
in the same mixed line graph. Finally, section \ref{sec:linegraphorient} will investigate 
under which conditions an arbitrary mixed orientation of some (undirected) line graph is actually
the mixed line graph of some mixed orientation of the root of the undirected line graph.

\section{Incidence matrix vs.\ line graph}\label{sec:line}

In this section we shall resolve the shortcomings mentioned in the introduction.
In the following, whenever we consider adjacency or incidence matrices of some graph, we tacitly refer to a
fixed (but otherwise arbitrary) vertex (resp.\ edge) order. The same reference order is also assumed when 
constructing derived objects (e.g.\ other matrices) by iterating over the vertex (resp.\ edge) set.
For the sake of convenience, we may index the entries of vertex order dependent matrices by the vertices themselves, not by
column/row numbers. For example, given a matrix $M$ and two vertices $x$ and $y$ such that $x$ is indexed as $i$-th vertex and
$y$ is indexed as $j$-th vertex, we use $M_{x,y}$ to refer to the cell at position $(i,j)$. Even shorter, $M_x$ refers to a diagonal entry. 
In the same spirit, we  use mixed vertex/edge indexing for incidence matrices.

Returning to equation \eqref{eq:btb}, we see that for a complex matrix $B$ the left-hand side of the equation is Hermitian.
Thus, it is natural to use a type of adjacency matrix of mixed graphs that is
Hermitian. Several authors have observed that there exists a natural generalization
of adjacency matrices as follows (cf.\ \cites{BM1,Liu}):

\begin{definition}\label{halpha}
Given a mixed graph $\gr{D}$ and	a number $\alpha\in\CC\setminus\RR$ with $\left\vert \alpha \right\vert =1$, we define the $\alpha$\hyp{}Hermitian adjacency matrix $H^\alpha(\gr{D})$ of $\gr{D}$ by
\begin{align}
(H^\alpha(\gr{D}))_{u,v} = \begin{cases}
1 & \text{if there is a digon between $u$ and $v$,}\\
\alpha &  \text{if there is an arc from $u$ to $v$,}\\
\bar{\alpha} &  \text{if there is an arc from $v$ to $u$,}\\
0 & \text{otherwise}.
\end{cases}
\end{align}
\end{definition}

When there is no ambiguity regarding the reference graph $\gr{D}$ we might just write $H^\alpha$ instead of $H^\alpha(\gr{D})$. 
The $\alpha$\hyp{}Hermitian adjacency matrix will be one cornerstone of what follows.
Inspired by its structure, we define an incidence matrix as follows:

\begin{definition}\label{balpha}
Given a mixed graph $\gr{D}$ and	a number $\beta\in\CC\setminus\RR$ with $\left\vert \beta \right\vert =1$, we define the $\beta$\hyp{}incidence matrix $B^\beta(\gr{D})$ of $\gr{D}$ by
\begin{align}
(B^\beta(\gr{D}))_{u,e} = \begin{cases}
1 & \text{if $e$ is a digon and $u$ is incident with it,}\\
\beta & \text{if $e$ is an arc and $u$ is its head vertex,}\\
\bar{\beta} &  \text{if $e$ is an arc and $u$ is its tail vertex,}\\
0 & \text{otherwise}.
\end{cases}
\end{align}
\end{definition}

Note that, both for $H^\alpha$ and $B^\beta$, the restriction $\alpha,\beta\not\in\RR$ is important in order to faithfully encode the
adjacency and incidence relations of mixed graphs.

Up to this point, the choices of $\alpha$ and $\beta$ can be made independently. 
However, requiring equation \eqref{eq:bbt} to be valid, restricts our choices drastically.

\begin{theorem}\label{thm:alphabeta}
Let $\gr{D}$ be a mixed graph, $B=B^\beta(\gr{D})$, $A=H^\alpha(\gr{D})$ and $D=\diag((\deg(v))_{v\in V(\gr{D})})$ the degree diagonal matrix of $\gr{D}$.
Then: 
\begin{enumerate}
\item[(i)]
$B B^\ast$ can be derived from $A + D$ be replacing each entry $\alpha$ (resp.\ $\overline{\alpha}$) 
with $\overline{\beta}^2$ (resp.\ $\beta^2$), and vice versa.
\item[(ii)]
Assuming $\gr{D}$ contains at least one arc, 
$B B^\ast = A + D$ if and only if $\alpha = \overline{\beta}^2$.
\end{enumerate}
\end{theorem}

\begin{proof}
This is a direct consequence of \Cref{{balpha}}.
To determine $(B B^\ast)_{u,v}$ one computes the inner product of row $u$ and the conjugate of row $v$ of $B$.
For $u=v$, since $\beta\overline{\beta}=1$, we obtain the degree of $u$. 
For $u\not=v$, $(B B^\ast)_{u,v}\not=0$ if and only if $u$ and $v$ are adjacent in $\gr{D}$.
In case of a digon $uv$ we have $(B B^\ast)_{u,v} = (A + D)_{u,v} = 1$.
In case of an arc from $u$ to $v$ (resp.\ $v$ to $u$) we have $(B B^\ast)_{u,v} = \alpha$ and $(A + D)_{u,v} = \overline{\beta}^2$
(resp.\ $(B B^\ast)_{u,v} = \overline{\alpha}$ and $(A + D)_{u,v} = {\beta}^2$).
Thus, $B B^\ast$ and $A + D$ have identical diagonals and the same zero-nonzero pattern.
Assuming $\gr{D}$ contains at least one arc, 
the two matrices are identical if and only if $\alpha = \overline{\beta}^2$.
\end{proof}

If we require identical parameters $\alpha=\beta$ for both the adjacency and the incidence matrix, then
this leaves us with just two choices:

\begin{corollary}\label{cor:btbchoice}
Let $\gr{D}$ be a mixed graph containing at least one arc, $B=B^\alpha(\gr{D})$, $A=H^\alpha(\gr{D})$ and $D$ the degree diagonal matrix of $\gr{D}$.
Then $B B^\ast = A + D$ 
if and only if $\alpha=\overline{\beta}^2\in\{\gamma, \gamma^2\}$, with $\gamma=e^{\frac{2\pi}{3}i}$.
\end{corollary}

Given any mixed graph $\gr{D}$, let $\Gamma(\gr{D})$ denote the underlying undirected graph, (i.e.\ derived by turning of arcs of $\gr{D}$ into digons).
By $A(\gr{G})$ we denote the traditional adjacency matrix of the undirected graph $\gr{G}$.
We now consider the goal equation \eqref{eq:btb}:

\begin{theorem}\label{thm:alphabeta2}
Let $\gr{D}$ be a mixed graph and $B=B^\beta(\gr{D})$.
Then: 
\begin{enumerate}
\item[(i)]
$B^\ast B$ has the same zero-nonzero pattern and main diagonal as the matrix $A(L(\Gamma(\gr{D})))+2I$.
\item[(ii)] 
Assuming $\gr{D}$ contains at least one vertex with both an incoming and an outgoing arc, 
there exists a mixed orientation $\gr{Y}$ of $L(\Gamma(\gr{D}))$ such that $B^\ast B=H_\alpha(\gr{Y})+2I$
if and only if either $\beta=\overline{\alpha}, \beta^2=\alpha$ or $\beta=\alpha, \beta^2=\overline{\alpha}$.
\end{enumerate}
\end{theorem}

\begin{proof}
Again check the consequences of \Cref{{balpha}}. 
To determine $(B^\ast B)_{e_1,e_2}$ one computes the inner product of the conjugated column $e_1$ and the column $e_2$ of $B$.
For $e_1=e_1$ one obtains either a term $1+1$ or $\alpha\overline{\alpha}+\overline{\alpha}\alpha$, thus $(B^\ast B)_{e_1,e_1}=2$.
For $e_1\not=e_2$, we have $(B^\ast B)_{e_1,e_2}\not=0$ if and only $e_1$ and $e_2$ are incident in $\gr{D}$. Moreover,
the value of $(B^\ast B)_{e_1,e_2}$ arises from a single nonzero term. This matches the situation
arising in equation \eqref{eq:btb} for the undirected graph $\Gamma(\gr{D})$. Thus, 
$B^\ast B$, $A(L(\Gamma(\gr{D})))$ and $H_\alpha(\gr{Y})$ have the same zero-nonzero pattern, for every mixed orientation $\gr{Y}$ of $L(\Gamma(\gr{D}))$.
It follows from \Cref{{balpha}} that $(B^\ast B)_{e_1,e_2}\in\{\beta, \overline{\beta}, 1=\beta\overline{\beta}=\overline{\beta}\beta, \beta^2, \overline{\beta}^2 \}=:S$.
The terms $\beta^2, \overline{\beta}^2$ arise whenever, at the common vertex of $e_1$ and $e_2$, edge $e_1$ an outgoing arc and $e_2$
is an incoming arc -- or vice versa. $B^\ast B-2I$ is an $\alpha$\hyp{}Hermitian adjacency matrix if and only if
$S\subseteq\{1,\alpha,\overline{\alpha}\}$. Since we require $\alpha,\beta\not\in\RR$, there exist only two viable mappings,
namely $\beta=\overline{\alpha}, \beta^2=\alpha$ and $\beta=\alpha, \beta^2=\overline{\alpha}$.
\end{proof}

Again, if we require identical parameters $\alpha=\beta$ for both the adjacency and the incidence matrix, then
this leaves us with just two choices:

\begin{corollary}\label{cor:bbtchoice}
Let $\gr{D}$ be a mixed graph containing at least one vertex with both an incoming and an outgoing arc, 
$B=B^\alpha(\gr{D})$ and $A=H^\alpha(\gr{D})$.
Then $B^\ast B=H_\alpha(\gr{Y})+2I$ for some mixed orientation $\gr{Y}$ of $L(\Gamma(\gr{D}))$ 
if and only if $\alpha=\overline{\beta}^2\in\{\gamma, \gamma^2\}$, with $\gamma=e^{\frac{2\pi}{3}i}$.
\end{corollary}

As we can see from \Cref{cor:btbchoice,cor:bbtchoice}, there exist only two natural parameter choices
to make equations \eqref{eq:btb} and \eqref{eq:bbt} work if use the matrices from \Cref{halpha,balpha}.
Thus, the following line graph construction results as a natural consequence:

{
\newcommand{\rootFig}[6]{
	\scalebox{0.6}{\begin{tikzpicture}\node[] () at (0,-5) {};
  \node[mynode] (v0) at (0.7954545454545454,-5.5681818181818175) {$u$};
  \node[mynode] (v1) at (2.4999999999999996,-5.5681818181818175) {$v$};
  \node[mynode] (v2) at (4.204545454545454,-5.5681818181818175) {$w$};
  \draw[#1] (#2) edge (#3) {};
  \draw[#4] (#5) edge (#6) {};
  \end{tikzpicture}}
}
\newcommand{\lgFig}[3]{
	\scalebox{0.6}{\begin{tikzpicture}\node[] () at (0,-5) {};
  \node[mynode] (v0) at (0.7954545454545454,-5.5681818181818175) {$uv$};
  \node[mynode] (v1) at (2.4999999999999996,-5.5681818181818175) {$vw$};
  \draw[#1] (#2) edge (#3) {};
  \end{tikzpicture}}
}
\begin{figure}
  \begin{center}
  \begin{tabular}{|c|c|} \hline
	in $\gr{D}$ & in $\AL(\gr{D})$ \\ \hline \hline
	\rootFig{myarc}{v0}{v1}{myarc}{v1}{v2} & \lgFig{myarc}{v0}{v1}  \\ \hline
	\rootFig{myarc}{v1}{v0}{myarc}{v1}{v2} & \lgFig{myedge}{v0}{v1}  \\ \hline
	\rootFig{myarc}{v0}{v1}{myarc}{v2}{v1} & \lgFig{myedge}{v0}{v1}  \\ \hline
	\rootFig{myarc}{v0}{v1}{myedge}{v1}{v2} & \lgFig{myarc}{v1}{v0}  \\ \hline
	\rootFig{myarc}{v1}{v0}{myedge}{v1}{v2} & \lgFig{myarc}{v0}{v1}  \\ \hline
	\rootFig{myedge}{v1}{v0}{myedge}{v1}{v2} & \lgFig{myedge}{v0}{v1}  \\ \hline
	\end{tabular}
	\end{center}
		\caption{Construction of the $\gamma$\hyp{}line graph $\AL(\gr{D})$}
	\label{fig:lgamma_constr}
\end{figure}
}

\begin{figure}
  \begin{center}
	\raisebox{0mm}{\scalebox{0.68}{
\begin{tikzpicture}
\node[mynode] (v0) at (1.4411764705882353,2.264705882352941) {0};
\node[mynode] (v1) at (3.5,2.264705882352941) {1};
\node[mynode] (v2) at (3.5,0.20588235294117646) {2};
\node[mynode] (v3) at (5.5588235294117645,2.264705882352941) {3};
\node[mynode] (v4) at (1.4411764705882353,0.20588235294117646) {4};
\node[mynode] (v5) at (5.5588235294117645,0.20588235294117646) {5};
\node[mynode] (v6) at (5.5588235294117645,-1.8529411764705883) {6};
\draw[myarc] (v0) edge (v1) {};
\draw[myarc] (v2) edge (v0) {};
\draw[myarc] (v0) edge (v4) {};
\draw[myedge] (v1) edge (v2) {};
\draw[myarc] (v3) edge (v2) {};
\draw[myarc] (v4) edge (v2) {};
\draw[myarc] (v5) edge (v2) {};
\draw[myedge] (v3) edge (v5) {};
\draw[myedge] (v5) edge (v6) {};
\end{tikzpicture}
}
}
\hspace*{5mm}\raisebox{16mm}{\scalebox{1.25}{$\stackrel{\AL}{\longrightarrow}$}}\hspace*{5mm}
\raisebox{0mm}{\scalebox{0.6}{
\begin{tikzpicture}
\node[mynode] (v0) at (4.092592592592593,2.648148148148148) {0-1};
\node[mynode] (v1) at (1.6851851851851851,0.24074074074074073) {0-4};
\node[mynode] (v2) at (7.978835978835979,0.24074074074074073) {1-2};
\node[mynode] (v3) at (4.092592592592593,0.24074074074074073) {0-2};
\node[mynode] (v4) at (8.907407407407407,2.648148148148148) {2-3};
\node[mynode] (v5) at (11.314814814814815,0.24074074074074073) {3-5};
\node[mynode] (v6) at (4.092592592592593,-2.1666666666666665) {2-4};
\node[mynode] (v7) at (8.907407407407407,-2.1666666666666665) {2-5};
\node[mynode] (v8) at (11.314814814814815,-2.1666666666666665) {5-6};
\draw[myedge] (v0) edge (v1) {};
\draw[myarc] (v2) edge (v0) {};
\draw[myarc] (v3) edge (v0) {};
\draw[myarc] (v3) edge (v1) {};
\draw[myarc] (v1) edge (v6) {};
\draw[myarc] (v3) edge (v2) {};
\draw[myarc] (v2) edge (v4) {};
\draw[myarc] (v2) edge (v6) {};
\draw[myarc] (v2) edge (v7) {};
\draw[myarc] (v4) edge (v3) {};
\draw[myarc] (v6) edge (v3) {};
\draw[myarc] (v7) edge (v3) {};
\draw[myarc] (v4) edge (v5) {};
\draw[myedge] (v4) edge (v6) {};
\draw[myedge] (v4) edge (v7) {};
\draw[myarc] (v7) edge (v5) {};
\draw[myedge] (v5) edge (v8) {};
\draw[myedge] (v6) edge (v7) {};
\draw[myarc] (v7) edge (v8) {};
\end{tikzpicture}
}
}
	\end{center}
		\caption{$\gamma$\hyp{}line graph example}
	\label{fig:lgamma_example}
\end{figure}
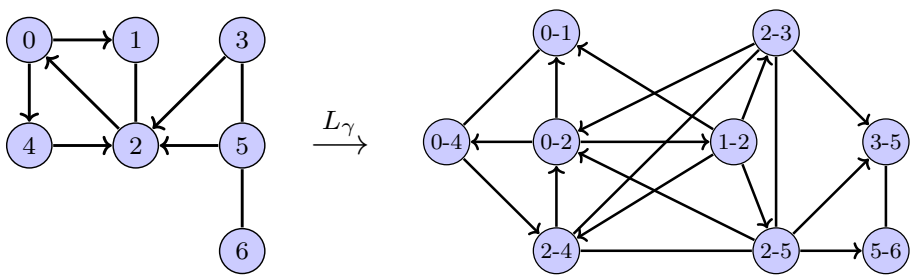

\begin{definition}\label{def:lgamma}
Given a mixed graph $\gr{D}$, the $\gamma$\hyp{}line graph $\AL(\gr{D})$ is defined as the mixed
orientation of $L(\Gamma(\gr{D}))$ arising according to \Cref{fig:lgamma_constr}.
\end{definition}

\Cref{fig:lgamma_example} contains an example illustrating \Cref{def:lgamma}.
As an alternative to \Cref{def:lgamma}, one can define the $\gamma^2$\hyp{}line graph of a mixed graph by reversing all arcs in the right column of
\Cref{fig:lgamma_constr}. In the following, we shall only be concerned with $\gamma$\hyp{}line graphs, but results for
$\gamma^2$\hyp{}line graphs can be obtained in a similar manner.

In view of \Cref{def:lgamma,cor:btbchoice,cor:bbtchoice}, the following result is now evident:

\begin{theorem}\label{thm:incidmatrixprod}
Let $\gr{D}$ be a mixed graph, $B=B^\gamma(\gr{D})$, $D=\diag((\deg(v))_{v\in V(\gr{D})})$. Then,
\begin{enumerate}
\item[(i)] $B^\ast B= \HG(\AL(\gr{D})) + 2I$,
\item[(ii)] $B B^\ast = \HG(\gr{D}) + D$.
\end{enumerate}
\end{theorem}

The following \Cref{thm:identincidence,thm:identincidroot} act as converses
to \Cref{thm:incidmatrixprod}.

\newcommand{\someIncid}{{R}} 

\begin{theorem}\label{thm:identincidence}
Let $\gr{X}$ be a mixed orientation of some graph $\gr{G}$. Let $\someIncid$ be matrix with nonzero entries from the
set $\{1, \gamma, \gamma^2\}$ having the same zero-nonzero pattern as the incidence matrix of $\gr{G}$.
If 
\begin{align}
\begin{aligned}\label{eq:identbp}
 \someIncid \someIncid^\ast & = \HG(\gr{X}) +  \diag((\deg(v))_{v\in V(\gr{T})}),
\end{aligned}
\end{align}
then $\someIncid=B^\gamma(\gr{X})$, i.e.\ it is the $\gamma$\hyp{}incidence matrix of $\gr{X}$.
\end{theorem}

\begin{proof} 
        Let entry $(i,j)$ of $\HG(\gr{X})$ be $\gamma$. Consider inner product of
        row $i$ of $\someIncid$ and row $j$ of $\someIncid^\ast$. Since $\someIncid$ is a $\gamma$\hyp{}incidence
        matrix the inner product contains exactly one non-zero term,
        either of the form $1 \cdot 1$, $\gamma \cdot \gamma$, or $\gamma^2 \cdot \gamma^2$.
        Only in the latter case the equation \eqref{eq:identbp} is true, hence entry $(i,ij)$ of $\someIncid$
        must be $\gamma^2$, so the edge $ij$ of $\gr{G}$ is oriented in $\gr{X}$ just as claimed.
        The remaining cases follow in the same straightforward manner.
\end{proof}

\begin{theorem}\label{thm:identincidroot}
Let $\gr{Y}$ be a mixed orientation of the mixed line graph $\AL(\gr{G})$ of some graph $\gr{G}$.
If $\someIncid$ is the $\gamma$\hyp{}incidence matrix of some mixed orientation $\gr{X}$ of $\gr{G}$ such that
\begin{align}
\begin{aligned}\label{eq:verifybpy}
\someIncid^\ast\someIncid & = \HG(\gr{Y}) + 2I,
\end{aligned} 
\end{align}
then $\gr{Y}=\AL(\gr{X})$, i.e.\ $\gr{X}$ is a root of $\gr{Y}$.
\end{theorem}

\begin{proof}
Consider entry $(i,j)$ of $\HG(\gr{Y})$ and assume it equals $\gamma$. 
It must match the inner product of column $i$ of $\someIncid^\ast$ and row $j$ of $\someIncid$.
By the structure of $\someIncid$, a only nonzero term in this product can only arise if edge $i$ is adjacent to edge $j$.
The only possible outcomes are $\gamma^2 \cdot \gamma^2$, $\gamma \cdot 1$, $1 \cdot \gamma$. They correspond
to the cases $x\stackrel{i}{\rightarrow} y \stackrel{j}{\text{---}} z$, $x \stackrel{i}{\leftarrow} y \stackrel{j}{\text{---}} z$, $x \stackrel{i}{\text{---}} y \stackrel{j}{\rightarrow} z$, respectively. Hence the edge $ij$
was locally created by a $\gamma$\hyp{}line graph operation. 
The remaining cases follow in the same straightforward manner.
\end{proof}

Although providing \Cref{thm:incidmatrixprod} has been the initial goal for the development of $\gamma$\hyp{}line graphs,
the notion of $\gamma$\hyp{}line graphs occurs to be a natural choice to the degree that
other well-known results from the domain of undirected line graphs carry over easily. 
The following \namecref{thm:alcharpoly} on characteristic polynomials $\chi$ is but one such result:

\begin{theorem}\label{thm:alcharpoly}
Let $\gr{X}$ be a mixed graph. If $\AL(\gr{X})$ is $k$-regular with $n$ vertices  and $m$ edges, then
\begin{align}
   \chi(\HG(\AL(\gr{X})), \lambda) = (\lambda + 2)^{m-n} \chi(\HG(\gr{X}), \lambda + 2 - k).
\end{align}
\end{theorem}

\begin{proof}
Similar to the proof of Theorem 3.8 in \cite{Biggs}.
\end{proof}

\section{Root orientations}\label{sec:root}

Given a line graph $H$, an (undirected) graph $\gr{G}$ satisfying $L(\gr{G})=H$ is called a root of $H$.
Likewise, given a $\gamma$\hyp{}line graph $\gr{Y}$, a mixed graph $\gr{X}$ satisfying $\AL(\gr{X})=\gr{Y}$ is called a root of $\gr{Y}$.
According to Whitney's Isomorphism Theorem (cf.\ \cite{Whitney}), any (undirected) line graph usually has exactly one root. The only exception is $C_3$, which has two roots,
namely itself and the star $K_{1,3}$. Therefore, given a $\gamma$\hyp{}line graph $\gr{Y}$, we immediately know 
all graphs $\gr{G}$ such that $L(\gr{G})=\Gamma(\gr{Y})$, i.e.\ the underlying graphs of all roots of $\gr{Y}$.
What remains to be decided is which mixed orientations $\gr{X}$ of such an undirected root $\gr{G}$ actually satisfy $L(\gr{X})=\gr{Y}$.

In view of \Cref{thm:identincidroot}, the equation $B^\ast B=\HG(\gr{Y})+2I$ can be used to derive necessary conditions
on the sought mixed orientation $\gr{X}$, in the sense that the initial choice of orienting any edge of $\gr{G}$ -- by way of propagating the
conditions along the paths of a spanning tree containing that edge -- necessarily determines all other mixed edge orientations. To this
end, we present the following construction:

\begin{construction}\label{con:mixroot}
Let $\gr{Y}$ be a $\gamma$\hyp{}line graph and let $\gr{G}$ be a root of $\Gamma(\gr{Y})$.
Construct a matrix $B$ as follows:
\begin{enumerate}
\item Initialize $B=0$.
\item\label{item:edgeini} Fix any initial edge $uv$ of $\gr{G}$ and arbitrarily set
$B_{u,uv}=\overline{B_{v,uv}}\in\{\gamma, \gamma^2, 1\}$.
\item Construct a spanning tree $\gr{T}$ of $\gr{G}$ containing the edge $uv$.
\item\label{item:mainstep} Let $u_1u_2$ be an edge whose mixed orientation has already been decided, i.e.\ $B_{u_1,u_1u_2}\not= 0 \not= {B_{u_2,u_1u_2}}$,
and let $u_2u_3$ be an adjacent edge not yet decided upon, i.e.\ $B_{u_2,u_2u_3}=0={B_{u_3,u_2u_3}}$.
  \begin{enumerate}
  \item Set $B_{u_2,u_2u_3}:=B_{u_2,u_1u_2}(\HG(\gr{Y}))_{u_2,u_3}$.
	\item Then set $B_{u_3,u_2u_3}:=\overline{B_{u_2,u_2u_3}}$.
  \end{enumerate}
\item Repeat step \ref{item:mainstep} until all edges of $\gr{T}$ have been decided upon.
\end{enumerate}
\end{construction}

\begin{theorem}\label{thm:rootconstr}
Let $\gr{Y}$ be a $\gamma$\hyp{}line graph and let $\gr{G}$ be a root of $\Gamma(\gr{Y})$.
Fix any initial edge $uv$ of $\gr{G}$ and 
use \Cref{con:mixroot} to construct a reference output matrix $B$. Then the following statements are equivalent:
\begin{enumerate}
\item[(i)]
The matrix $B$ can be augmented such that it constitutes the $\gamma$\hyp{}incidence matrix of some
mixed orientation $\gr{X}$ of $\gr{G}$ such that $\AL(\gr{X})=\gr{Y}$ and $uv$ has the chosen initial
orientation in $\gr{X}$.
\item[(ii)]Constructing the set of all output matrices of \Cref{con:mixroot} for the graph $\gr{G}$ such that
the two initial entries defined in step \ref{item:edgeini} agree with the corresponding entries in the reference matrix $B$,
one finds that, whenever one of the output matrices has a non-zero entry, all output matrices have the same
non-zero entry in that position.
\end{enumerate}
\end{theorem}

\begin{proof}
Each output matrix created in (ii) conveys a partial result on how to necessarily orient a mixed orientation $\gr{X}$ of $\gr{G}$
such that $\AL(\gr{X})=\gr{Y}$ such that the initial edge is oriented as chosen. 
Only if all partial results agree with one another with respect to the
obtained non-zero entries is it possible to merge them into the $\gamma$\hyp{}incidence matrix of the only possible mixed orientation $\gr{X}$ of $\gr{G}$
that can have the desired properties. However, one needs to verify that $B^\ast B=\HG(\gr{Y})+2I$ is actually satisfied.
To this end, step \ref{item:mainstep} of \Cref{con:mixroot} only ensures that the inner product of the conjugated row $u_2$ and the row $u_3$ yields 
a valid partial result towards matching the goal equation. In general, the procedure does not verify every single case of edge adjacency in $\gr{X}$, i.e.\
it does not utilize every relevant single inner product formed 
in the product $B^\ast B$. However, considering all of the inner products considered for all output matrices that have been generated as described,
it follows that $B^\ast B=\HG(\gr{Y})+2I$ has been fully verified. The converse direction is obvious.
\end{proof}

\begin{corollary}
Any $\gamma$\hyp{}line graph $\gr{Y}$ has at most 3 mixed root orientations $\gr{X}$ such that $\AL(\gr{X})=\gr{Y}$.
\end{corollary}

\begin{theorem}\label{thm:bip_three_roots}
Let $\gr{Y}$ be a $\gamma$\hyp{}line graph and let $\gr{G}$ be a root of $\Gamma(\gr{Y})$.
If $\gr{G}$ is bipartite, then there exist exactly three mixed root orientations $\gr{X}$ such that $\AL(\gr{X})=\gr{Y}$, otherwise there exists exactly one such orientation.
\end{theorem}

\begin{proof}
Changing the choice of mixed orientation for the initial edge effectively means multiplying $B_{u,uv}$ by $\gamma$ (resp.\ $\gamma^2$) 
and $B_{v,uv}$ by $\gamma^2$ (resp.\ $\gamma$).
As a consequence, in view of \Cref{con:mixroot}, any edge $e_{v,w}$ needs to adjust its orientation accordingly, effectively
multiplying $B_{v,vw}$ by $\gamma$ (resp.\ $\gamma^2$).
Thus, the entire row $v$ gets multiplied by this factor. Moreover, along with $B_{v,vw}$
we need to adjust $B_{w,vw}$ and this multiply it by $\gamma^2$ (resp.\ $\gamma$), in order to retain $B$ as a $\gamma$\hyp{}incidence matrix. 
Considering the propagation process along the branches of the spanning tree $\gr{T}$, away from the
initial edge, we see that all rows $y$ corresponding of vertices $y$ having odd distance from $u$ in $\gr{T}$ need to be multiplied 
by $\gamma$ (resp.\ $\gamma^2$), whereas all rows corresponding to even distance vertices need to be multiplied by $\gamma^2$ (resp.\ $\gamma$) 
-- as a necessary adjustment. Choosing a different spanning tree might, however, result in different correction factors.
With respect to \Cref{thm:rootconstr}, we see that the procedure described above yields a valid mixed root orientation if any only if
the graph $\gr{G}$ is bipartite (regardless of which of the two options for reorientation we chose at the start), i.e.\ if it does not contain any
odd cycles.
\end{proof}

\begin{corollary}
Let $\gr{Y}$ be a $\gamma$\hyp{}line graph and let $\gr{G}$ be a root of $\Gamma(\gr{Y})$.
Given two mixed orientations $\gr{X}$ and $\gr{X}'$ of $\gr{G}$ such that $\AL(\gr{X})=\AL(\gr{X}')=\gr{Y}$,
there exists a diagonal matrix $\gr{D}$ with diagonal entries $\gamma, \gamma^2$ only such
that $DB=B'$ (where $B$ and $B'$ are the $\gamma$\hyp{}incidence matrices of $\gr{X}$ and $\gr{X}'$, respectively).
\end{corollary}

\section{Line graph orientations}\label{sec:linegraphorient}

Next, we shall shift our perspective a little. Given an arbitrary mixed orientation $\gr{Y}$ of some line graph $H=L(\gr{G})$, we will now
investigate under which conditions we can guarantee that there exists a mixed orientation $\gr{X}$ of $\gr{G}$ such that $L(\gr{X})=\gr{Y}$. First let us recall
the following \namecref{def:cliq}:

\begin{definition}\label{def:cliq}
Given an undirected graph $\gr{G}$, a system $Q_1,\ldots,Q_k$ of cliques of $\gr{G}$ is called a complete system of cliques if the following conditions are satisfied:
\begin{itemize}
\item[(i)] For $i\not= j$ we have $\vert Q_i \cap Q_j\vert \leq 1$.
\item[(ii)] Every vertex of $\gr{G}$ is contained in exactly two of the cliques.
\item[(iii)] If $\vert Q_i \cap Q_j\vert = 1$, then $\vert Q_i\vert + \vert Q_j\vert = \deg(u)+2$, where $\{u\}=Q_i \cap Q_j$.
\end{itemize}
\end{definition}

Note that \Cref{def:cliq} permits trivial cliques, i.e.\ some cliques may be isomorphic to $K_1$.
Next we state a classic characterization of line graphs:

\begin{theorem}[cf.\ \cite{Ray}]\label{thm:cliquesys}
An undirected graph is a line graph if and only if it admits a complete system of cliques.
\end{theorem}

\begin{definition}\label{hwalk}
	Let $\gr{D}$ be a mixed graph and $H=H^\alpha(\gr{D})$. With respect to $\gr{D}$ and $H$, 
  the value $\hal(W)$ of a mixed walk $W$ with vertices $v_1,v_2,\ldots,v_k$ is defined as
		\begin{align}
		\hal(W) = (H_{v_1v_2}H_{v_2v_3}H_{v_3v_4}\cdots H_{v_{k-1}v_k}) \in\{\alpha^r\}_{r\in\mathbb{Z}}.
		\end{align}
\end{definition}

\begin{theorem}\label{thm:weight1}
Let $\gr{X}$ be a mixed graph and $\gr{Y}=\AL(\gr{X})$. Further, let $Q_1,\ldots,Q_k$ be a complete system of cliques of $\Gamma(\gr{Y})$ and
let $C$ be a cycle in $Q_i$ (for some $i\in\{1,\ldots,k\}$). Then the mixed cycle $\vec{C}$ in $\gr{Y}$ that corresponds to $C$ 
has weight $\hal(\vec{C})=1$ in $\gr{Y}$.
\end{theorem}

\begin{proof}
Under the assumptions of the theorem, the clique $Q_i$ of $\gr{Y}$ corresponds to the star subgraph induced by the edges incident some vertex $r$ in $\Gamma(\gr{X})$.
Suppose that $C$ is traversed by $rs_1,rs_2,\ldots, rs_m, rs_1$ (where $s_1,\ldots,s_m\in V(\gr{X})$). Let $\vec{C}$ be the mixed cycle in $\gr{Y}$ that
corresponds to $C$.
Using any traversal direction, we compute the weight of $\vec{C}$ as
\begin{align}
\begin{aligned}
\hal(\vec{C}) & = \HG(\gr{Y})_{rs_1,rs_2}\HG(\gr{Y})_{rs_2,rs_3}\cdots \HG(\gr{Y})_{rs_m,rs_1} \\
           & = ([B]_{rs_1}^\ast[B]_{rs_2})([B]_{rs_2}^\ast[B]_{rs_3})\cdots ([B]_{rs_m}^\ast[B]_{rs_1}) \\
					 & = \overline{B_{r,rs_1}}B_{r,rs_2}\overline{B_{r,rs_2}}B_{r,rs_3}\cdots \overline{B_{r,rs_m}}B_{r,rs_1} 
					 = 1,
\end{aligned}
\end{align}
where $[B]_e$ denotes the column of the incidence matrix $B$ of $\gr{X}$ that corresponds to the edge $e$.
\end{proof}

\begin{example}
When checking whether some mixed orientations of a given line graph $L(\gr{G})$ is actually the mixed line graph of some mixed orientation
of the root $\gr{G}$, the necessary condition stated in \Cref{thm:weight1} can be used as a first check.
For example, the mixed graph shown in \Cref{fig:invalidmix2} fails to satisfy the condition (cf.\ the triangle with vertices 0-1,0-2,0-4).
\end{example}

\begin{figure}
  \begin{center}
	\begin{subfigure}{.4\textwidth}
		\centering
\raisebox{0mm}{\scalebox{0.68}{
\begin{tikzpicture}
\node[mynode] (v0) at (0.9459459459459459,4.1891891891891895) {0};
\node[mynode] (v1) at (4.324324324324325,4.1891891891891895) {1};
\node[mynode] (v2) at (4.324324324324325,1.4864864864864864) {2};
\node[mynode] (v3) at (7.702702702702703,0.13513513513513514) {3};
\node[mynode] (v4) at (0.9459459459459459,1.4864864864864864) {4};
\node[mynode] (v5) at (7.027027027027027,2.8378378378378377) {5};
\node[mynode] (v6) at (9.054054054054054,4.864864864864865) {6};
\draw[myedge] (v0) edge (v1) {};
\draw[myedge] (v0) edge (v2) {};
\draw[myedge] (v0) edge (v4) {};
\draw[myedge] (v1) edge (v2) {};
\draw[myedge] (v2) edge (v3) {};
\draw[myedge] (v2) edge (v4) {};
\draw[myedge] (v2) edge (v5) {};
\draw[myedge] (v3) edge (v5) {};
\draw[myedge] (v5) edge (v6) {};
\end{tikzpicture}
}
}
		\caption{Root graph $\gr{G}$}
		\label{fig:invalidmix1}
	\end{subfigure}\hspace*{.05\textwidth}
	\begin{subfigure}{.4\textwidth}
		\centering
\raisebox{0mm}{\scalebox{0.6}{
\begin{tikzpicture}
\node[mynode] (v0) at (2.2972972972972974,6.216216216216216) {0-1};
\node[mynode] (v1) at (0.9459459459459459,2.8378378378378377) {0-4};
\node[mynode] (v2) at (5.0,4.1891891891891895) {1-2};
\node[mynode] (v3) at (2.972972972972973,4.1891891891891895) {0-2};
\node[mynode] (v4) at (5.675675675675675,0.13513513513513514) {2-3};
\node[mynode] (v5) at (9.054054054054054,0.13513513513513514) {3-5};
\node[mynode] (v6) at (2.972972972972973,1.4864864864864864) {2-4};
\node[mynode] (v7) at (7.027027027027027,2.8378378378378377) {2-5};
\node[mynode] (v8) at (9.054054054054054,4.1891891891891895) {5-6};
\draw[myarc] (v0) edge (v1) {};
\draw[myarc] (v2) edge (v0) {};
\draw[myarc] (v3) edge (v0) {};
\draw[myedge] (v1) edge (v3) {};
\draw[myarc] (v1) edge (v6) {};
\draw[myedge] (v2) edge (v3) {};
\draw[myarc] (v2) edge (v4) {};
\draw[myarc] (v6) edge (v2) {};
\draw[myarc] (v2) edge (v7) {};
\draw[myarc] (v3) edge (v4) {};
\draw[myarc] (v6) edge (v3) {};
\draw[myarc] (v3) edge (v7) {};
\draw[myarc] (v5) edge (v4) {};
\draw[myarc] (v6) edge (v4) {};
\draw[myedge] (v4) edge (v7) {};
\draw[myarc] (v7) edge (v5) {};
\draw[myedge] (v5) edge (v8) {};
\draw[myarc] (v7) edge (v6) {};
\draw[myarc] (v7) edge (v8) {};
\end{tikzpicture}		
}
}
		\caption{Mixed orientation of $L(\gr{G})$}
		\label{fig:invalidmix2}
	\end{subfigure}
	\end{center}
	\caption{Mixed orientation of a line graph that is not a mixed line graph}
	\label{fig:invalidmix}
\end{figure}

As can be seen from \Cref{thm:weight1}, the existence of feasible root orientations can be linked
to the algebraic properties of the cycles in the candidate $\gamma$\hyp{}line graph.
The goal of the remainder of this section is to prove that, for a line graph $L(\gr{T})$ of a tree $\gr{T}$, it
suffices to have a mixed $\gamma$\hyp{}monograph orientation $\gr{Y}$ in order to guarantee the existence
of a mixed root orientation $\gr{X}$ of $\gr{T}$ such that $\AL(\gr{X})=\gr{Y}$. This will be shown in \Cref{thm:monoroot}.
In preparation of this \namecref{thm:monoroot}, we first require some auxiliary results on monographs.

\begin{definition}[cf.\ \cite{hspec}]\label{halphw}
	Let $\gr{D}$ be a mixed graph.
	\begin{enumerate}
	\item[(i)]
  $\gr{D}$ is called an $\alpha$\hyp{}monograph if $\hal(\vec{C})=1$ for all its cycles $C$, where
	$\vec{C}$ denotes an arbitrary closed traversal walk on $C$.
	\item[(ii)]
  The $\alpha$\hyp{}store $S^\alpha(u)$ of  $u\in V(\gr{D})$ is defined as
	\begin{align}
  S^\alpha(u)=\{ \hal(W):\ \text{$W$ is a closed walk in $\gr{D}$ from/to $u$}\}.
	\end{align}
	\end{enumerate}
\end{definition}

Monographs capture the idea of transporting values along the edges of a mixed graph. One starts by assigning a seed value 
to some initial vertex. Spreading along a forward arc, the value at its terminal vertex will be $\alpha$ times the value at its initial vertex.
For a backward arc, the factor is $\overline{\alpha}$. For a digon, the factor is $1$. The required factors are easily looked up
in the $\alpha$\hyp{}Hermitian adjacency matrix. Trivially, trees are $\alpha$\hyp{}monographs. 

Clearly, $1\in S^\alpha(u)$ so that $\vert S^\alpha(u)\vert\geq 1$.
It is not hard to see that the store content is independent of the reference vertex $u$, i.e.\ $S^\alpha(u) = S^\alpha(v)$
for any $u,v\in V(\gr{D})$. Therefore, we may simply speak of `the' $\alpha$\hyp{}store of $\gr{D}$.
Using the store idea, $\alpha$\hyp{}monographs can be characterized as follows:

\begin{theorem}[cf.\ \cite{hspec}]\label{stchar}
Let $\gr{D}$ be a connected mixed graph. Then the following statements are equivalent:
\begin{itemize}\item[]\begin{enumerate}[label=(\roman*)]
\item \label{stchar5} $\gr{D}$ is an $\alpha$\hyp{}monograph.
\item \label{stchar4} $\hal(W')=\hal(W'')$ for every pair $W',W''$ of mixed walks sharing the same start and end vertices.
\item \label{stchar2} $\vert S^\alpha(u)\vert=1$ for every $u\in V(\gr{D})$.
\end{enumerate}\end{itemize}
\end{theorem}

\begin{theorem}\label{thm:reorient1}
Let $\gr{X}$ be an $\alpha$\hyp{}monograph and $u\in V(\gr{X})$. Let $S_u$ be the associated store (using any seed value $S_u(u)=\alpha^k$) and
define the matrix $\OR=\diag((S_u(v))_{v\in V(\gr{X})})$. Then
\begin{align}
 \OR \HA(\gr{X}) \OR^\ast = A(\Gamma(\gr{X})).
\end{align}
\end{theorem}

\begin{proof}
Since $\OR$ is an invertible diagonal matrix, $\OR \HA(\gr{X}) \OR^\ast$ and $\HA(\gr{X})$ have the same zero-nonzero pattern. 
Therefore it suffices to prove that $\OR \HA(\gr{X}) \OR^\ast$ is a $0$-$1$-matrix.
Let $xy$ be any edge of $\alpha(\gr{X})$. Since $\gr{X}$ is an $\alpha$\hyp{}monograph, we have 
${\OR_x}/{\OR_y} = (\HA(\gr{X}))_{y,x} = \overline{(\HA(\gr{X}))_{x,y}}$
by the definition of the store $S_u$.
So we obtain 
$
    (\OR \HA(\gr{X}) \OR^\ast)_{x,y} 
	= \OR_x (\HA(\gr{X}))_{x,y} \overline{\OR}_y 
	= 1.
$
\end{proof}

\begin{remark}\label{rem:switching}
Application of \Cref{thm:reorient1} to $\gamma$\hyp{}monographs (analogously, to $\gamma^2$\hyp{}monographs) yields an edge switching procedure that
will turn any mixed graph $\gr{X}$ into its unoriented counterpart. Recall that $\gamma^2=\overline\gamma$.
Performing the multiplication $\OR \HG(\gr{X}) \OR^\ast$, for every vertex $x$
we effectively multiply its associated row in $\HG(\gr{X})$ by $S_u(x)$ and its associated column by $\overline{S_u(x)}$. 
Thus, edges adjacent to $x$ are subjected to the switching pattern depicted in \Cref{fig:switching}.
Subsequently applying this pattern to all vertices of $\gr{X}$ (in any order) yields $\Gamma(\gr{X})$.
\end{remark}

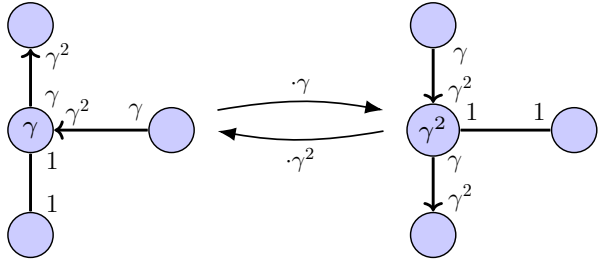
\begin{figure}
		\begin{center}
	\raisebox{0mm}{\scalebox{0.68}{
\begin{tikzpicture}
\node[mynode] (v0) at (0.9545454545454545,-3.2727272727272725) {$\gamma$};
\node[mynode] (v1) at (0.9545454545454545,-1.2272727272727273) {};
\node[mynode] (v2) at (3.6818181818181817,-3.2727272727272725) {};
\node[mynode] (v3) at (0.9545454545454545,-5.318181818181818) {};
\draw[myarc] (v0) edge (v1) {};
\draw[myarc] (v2) edge (v0) {};
\draw[myedge] (v0) edge (v3) {};
\draw[font=\Large] (v0) to node[xshift=15, very near end] {$\gamma^2$} 
                         node[xshift=12,very near start] {$\gamma$} (v1);
\draw[font=\Large] (v0) to node[yshift=10,very near end] {$\gamma$} 
                         node[yshift=10, near start] {$\gamma^2$} (v2);
\draw[font=\Large] (v0) to node[xshift=12,very near end] {$1$} 
                         node[xshift=12,very near start] {$1$} (v3);
\end{tikzpicture}
}
}
		\scalebox{0.8}{%
    \begin{tikzpicture}
\node (v0a) at (0,1) {};
\node (v1a) at (3,1) {};
\node (v0b) at (0,1.3) {};
\node (v1b) at (3,1.3) {};
\node (dummy) at (0,-1) {};
\draw[]  (v0b)  edge [-{Latex[length=3mm]},bend left =10,thick] node[midway,above] {$\cdot\gamma$} (v1b) ;
\draw[]  (v0a)  edge [{Latex[length=3mm]}-,bend right =10,thick] node[midway,below] {$\cdot\gamma^2$} (v1a) ;
				\end{tikzpicture}}
	\raisebox{0mm}{\scalebox{0.68}{
\begin{tikzpicture}
\node[mynode] (v0) at (0.9545454545454545,-3.2727272727272725) {$\gamma^2$};
\node[mynode] (v1) at (0.9545454545454545,-1.2272727272727273) {};
\node[mynode] (v2) at (3.6818181818181817,-3.2727272727272725) {};
\node[mynode] (v3) at (0.9545454545454545,-5.318181818181818) {};
\draw[myarc] (v1) edge (v0) {};
\draw[myedge] (v0) edge (v2) {};
\draw[myarc] (v0) edge (v3) {};
\draw[font=\Large] (v0) to node[xshift=15,very near end] {$\gamma$} 
                         node[xshift=15, near start] {$\gamma^2$} (v1);
\draw[font=\Large] (v0) to node[yshift=10,very near end] {$1$} 
                         node[yshift=10,very near start] {$1$} (v2);
\draw[font=\Large] (v0) to node[xshift=15, near end] {$\gamma^2$} 
                         node[xshift=12,very near start] {$\gamma$} (v3);
\end{tikzpicture}
}
}
			\caption{Pattern for switching an $\alpha$\hyp{}monograph into its undirected counterpart} 
			\label{fig:switching}
		\end{center}
\end{figure}

\begin{theorem}\label{thm:reorient2}
Let $\gr{G}$ be a graph. Further, let $\alpha\in\CC\setminus\RR$ with $\vert\alpha\vert=1$. Given any matrix $\OR=\diag((\OR_v)_{v\in V(\gr{X})})$ such that
$\OR_v = \alpha^k \OR_u$ for some $k\in\{-1,0,1\}$ is satisfied for all $uv\in E(\gr{G})$, we have:
\begin{itemize}
\item[(i)] $\OR^\ast A(\gr{G}) \OR$ is the $\alpha$\hyp{}adjacency matrix of a mixed graph $\gr{X}$,
\item[(ii)] $\Gamma(\gr{X})=G$,
\item[(iii)] $\gr{X}$ is an $\alpha$\hyp{}monograph.
\end{itemize}
\end{theorem}

\begin{proof}
Clearly, $H:=\OR^\ast A(\gr{G}) \OR$ is a Hermitian matrix with the same zero-nonzero-pattern as $A(\gr{G})$. By the condition imposed on $\OR$ it follows
that all nonzero entries in $H$ must be from the set $\{1, \alpha, \overline{\alpha}\}$. Hence, (i) and (ii) have been proven.
With respect to (iii), let $C=v_1v_2v_3\ldots v_kv_1$ be any cycle in $\gr{G}$. Then the weight of $C$ is
\begin{align}
  \hal(C) = H_{v_1,v_2} H_{v_2,v_3} \cdots H_{v_k,v_1}  = (\overline{\OR_{v_1}}\OR_{v_2})(\overline{\OR_{v_2}}\OR_{v_3}) \cdots (\overline{\OR_{v_{k}}}\OR_{v_1}) = 1.
\end{align}
\end{proof}

\begin{corollary}
Every graph has a nontrivial mixed orientation that yields an $\alpha$\hyp{}monograph.
\end{corollary}

In the case of $\gamma$\hyp{}monographs, the matrices $\OR$ encountered in \Cref{thm:reorient1,thm:reorient2} play an important role in describing 
the relation between a mixed root graph $\gr{G}$ and its mixed line graph $\AL(\gr{G})$.

\begin{definition}
Let $\gr{X}$ be a $\gamma$\hyp{}monograph. Any diagonal matrix $\OR$ with entries from the set $\{0, \gamma, \gamma^2\}$ satisfying
$\OR \HG(\gr{X}) \OR^\ast=A(\Gamma(\gr{X}))$ shall be called an orientation matrix of $\gr{X}$.
\end{definition}

\begin{theorem}
Let $\gr{X}$ be a $\gamma$\hyp{}monograph and $B$ its $\gamma$-incidence matrix. Further, let $\OR$ be an orientation matrix of $\gr{X}$. 
Define the matrix $\ORd = \diag((\ORd_{uv})_{uv\in E(\Gamma(\gr{X}))})$ by $\ORd_{uv}=\OR_u B_{u,uv}$. Then,
\begin{itemize}
\item[(i)] $\OR B \ORd^\ast$ is the incidence matrix of $\Gamma(\gr{X})$.
\item[(ii)] $\ORd$ is an orientation matrix of $\AL(\gr{X})$.
\end{itemize}
\end{theorem}

\begin{proof}
Since $\OR$ and $\ORd$ are invertible diagonal matrices, $B':=\OR B \ORd^\ast$ and $B$ have the same zero-nonzero pattern. 
Therefore, regarding claim (i), it suffices to prove that $B'$ is a $0$-$1$-matrix. Since for any $uv\in E(\gr{X})$ we have
\begin{align}
(\OR B \ORd^\ast)_{u,uv} = \OR_u B_{u,uv} (\ORd^\ast)_{uv} = \OR_u B_{u,uv} \overline{\OR_u} \overline{B_{u,uv}} = 1
\end{align}
it follows that this is indeed the case.

To prove that $\ORd$ is an orientation matrix of $\AL(\gr{X})$ we need to assert that $\ORd \HG(\AL(\gr{X})) \ORd^\ast = A(\Gamma(\AL(\gr{X})))$.
To this end, we verify
\begin{align}
\begin{aligned}
A(\Gamma(\AL(\gr{X}))) &= A(L(\Gamma(\gr{X}))) = (\OR B \ORd^\ast)^\ast(\OR B \ORd^\ast) - 2I  \\
                  &= \ORd B^\ast \OR^\ast \OR B \ORd^\ast - 2I = \ORd(B^\ast B - 2I)\ORd^\ast \\
									&= \ORd \HG(\AL(\gr{X})) \ORd^\ast.
\end{aligned}
\end{align}
\end{proof}

\begin{theorem}\label{thm:weight2}
Let $\gr{X}$ be a mixed graph and $\gr{Y}=\AL(\gr{X})$. Further, let $\vec{C}$ be a mixed cycle in $\gr{X}$ and $\vec{C'}$ its corresponding cycle in $\gr{Y}$.
Then $\hga(\vec{C}, \gr{X}) = \hga(\vec{C}, \gr{X})$, i.e.\ the weight of $\vec{C}$ in $\gr{X}$ and the weight of $\vec{C'}$ in $\gr{Y}$ are equal (using analogous traversal direction).
\end{theorem}

\begin{proof}
Assume that $\vec{C}$ is traversed as $u_1,u_2,\ldots,u_k,u_1$ (with $u_i\in V(\gr{X})$). 
Let $[B]_{u_i}$ denote the row of the incidence matrix $B$ of $\gr{X}$ that corresponds to the vertex $u_i$. Then,
\begin{align}\label{eq:weightvecc1}
\begin{aligned}
\hga(\vec{C}, \gr{X}) & = \HG(\gr{X})_{u_1,u_2}\HG(\gr{X})_{u_2,u_3}\cdots \HG(\gr{X})_{u_m,u_1} \\
           & = ([B]_{u_1}([B]_{u_2})^\ast)([B]_{u_2}([B]_{u_3})^\ast)\cdots ([B]_{u_m}([B]_{u_1})^\ast) \\
					 & = ((B)_{u_1,u_1u_2}\overline{(B)_{u_2,u_1u_2}})((B)_{u_2,u_2u_3}\overline{(B)_{u_3,u_2u_3}}) \cdots ((B)_{u_k,u_ku_1}\overline{(B)_{u_1,u_ku_1}}).
\end{aligned}
\end{align}
Let $[B^\ast]_{u_iu_j}$ denote the row of the matrix $B^\ast$ that corresponds to the edge $u_iu_j$.
Observe that 
\begin{align}\label{eq:bbbb}
  [B^\ast]_{u_iu_j}([B^\ast]_{u_iu_k})^\ast = (B^\ast B)_{u_iu_j,u_ju_k} = (\HG(\AL(\gr{X})))_{u_iu_j,u_ju_k}
\end{align}
for $u_i\not= u_j \not= u_k$.
Moving the first term in the final product of \eqref{eq:weightvecc1} to the back and making use of \eqref{eq:bbbb}, we get
\begin{align}\label{eq:weightvecc2}
\begin{aligned}
\hga(\vec{C}, \gr{X}) & = (\overline{(B)_{u_2,u_1u_2}}(B)_{u_2,u_2u_3})(\overline{(B)_{u_3,u_2u_3}(B)_{u_3,u_3u_4}} \cdots (\overline{(B)_{u_1,u_ku_1}} (B)_{u_1,u_1u_2}) \\
           & = ([B^\ast]_{u_1u_2}([B^\ast]_{u_2u_3})^\ast)([B^\ast]_{u_2u_3}([B^\ast]_{u_3u_4})^\ast) \cdots ([B^\ast]_{u_ku_1}([B^\ast]_{u_1u_2})^\ast) \\
					 & = (\HG(\AL(\gr{X})))_{u_1u_2,u_2u_3}(\HG(\AL(\gr{X})))_{u_2u_3,u_3u_4} \cdots (\HG(\AL(\gr{X})))_{u_ku_1,u_1u_2} \\
					 & = \hga(\vec{C'}, \gr{Y}).
\end{aligned}
\end{align}
\end{proof}

\begin{corollary}\label{thm:monoiff}
A mixed graph $\gr{X}$ is an $\alpha$\hyp{}monograph if and only if $\AL(\gr{X})$ is an $\alpha$\hyp{}monograph.
\end{corollary}

Given some mixed orientation $\gr{Y}$ of the line graph of a tree, a necessary condition for the existence of a mixed root is that all cycles in $\gr{Y}$ must 
satisfy the condition stated in \Cref{thm:weight1}. It turns out that this condition is actually sufficient:

\begin{theorem}\label{thm:monoroot}
Let $\gr{Y}$ be a mixed $\gamma$\hyp{}monograph such that $\Gamma(\gr{Y})=L(\gr{T})$ for some tree $\gr{T}$.
Then $\gr{Y}=\AL(\gr{X})$ for some mixed graph $\gr{X}$ with $\Gamma(\gr{X})=\gr{T}$.
\end{theorem}

\begin{proof}
Our goal is to construct a mixed orientation $\gr{X}$ of $\gr{T}$ such that $\gr{Y}=\AL(\gr{X})$. By \Cref{thm:reorient1}, since $\gr{Y}$ is a $\gamma$\hyp{}monograph,
it has an orientation matrix $\ORd$, i.e.
\begin{align}\label{eq:Yorient}
  A(\Gamma(\gr{Y})) = \ORd \HG(\gr{Y}) \ORd^\ast.
\end{align}

Below we outline a recursive procedure that defines a diagonal matrix $\OR$ and the
$\gamma$\hyp{}incidence matrix $\Bp$ of a mixed orientation of $\gr{T}$ satisfying the claim of the \namecref{thm:monoroot}.
To this end, let $B$ be the incidence matrix of $\gr{T}$. Fix some vertex $u\in V(\gr{T})$ and 
a seed value $\OR_u\in\{1,\gamma,\gamma^2\}$.
Now, consider any path $P$ in $\gr{T}$ with vertices $u=v_1, v_2, \ldots, v_k=v$ from $u$ to $v\in V(\gr{T})$
and use the defining equation
\begin{align}\label{eq:recur}
  \ORd_{v_iv_{i+1}} = \OR_{v_i}\Bp_{v_i, v_iv_{i+1}}
\end{align}
in order to determine $\Bp_{v_1,v_1v_2}$:
\begin{align}\label{eq:recur_step1}
  \Bp_{v_1,v_1v_2} = \frac{\ORd_{v_1v_2}}{\OR_{v_1}}.
\end{align}
Requiring $\Bp$ to be a $\gamma$\hyp{}incidence matrix, we set $\Bp_{v_2,v_1v_2}=\overline{\Bp_{v_1,v_1v_2}}$.
Using equation \eqref{eq:recur} once again, we deduce
\begin{align}\label{eq:recur_step2}
  \ORd_{v_1v_2} = \ORd_{v_2v_1} = \OR_{v_2} \Bp_{v_2,v_1v_2}
\end{align}
and therefore we can compute $\OR_{v_2}$ from previously known values as
\begin{align}\label{eq:recur_step3}
\OR_{v_2} = \overline{\ORd_{v_1v_2}}\ \overline{\OR_{u_1}}.
\end{align}
Continuing along the vertices of $P$, we can repeatedly apply the same pattern of computations
as in equations \eqref{eq:recur_step1}, \eqref{eq:recur_step2}, \eqref{eq:recur_step3} to get a
recursive formula
\begin{align}\label{eq:prefinalrecur}
\OR_{v_i} = (\ORd_{v_iv_{i-1}})^2 \overline{\OR_{v_{i-1}}}
\end{align}
and finally
\begin{align}\label{eq:finalrecur}
\OR_{v_i} = (\ORd_{v_iv_{i-1}})^2 (\overline{\ORd_{v_{i-1}v_{i-2}}})^2 \cdot \ldots \cdot \begin{cases} \OR_{v_1} & \text{if $i$ is odd} \\ \overline{\OR_{v_{1}}} & \text{if $i$ is even} \end{cases}.
\end{align}

Since every vertex $v$ of $\gr{T}$ can be reached by a unique path $P$ from $u$ to $v$ in $\gr{T}$, the matrix $\OR$ is thus complete and well-defined.
Minding the seed value, it is clear from \eqref{eq:finalrecur} that $\OR_{v_i}\in\{1,\gamma,\gamma^2\}$.
Hence, by virtue of \Cref{thm:reorient2}, we have
\begin{align}\label{eq:rootxdef}
 \OR^\ast A(\gr{T}) \OR = \HG(\gr{X})
\end{align}
for some mixed orientation $\gr{X}$ of $\gr{T}$.

With respect to the partially defined matrix $\Bp$, note that for every edge $v_1v_2$ of $\gr{T}$ we have defined two entries
$\Bp_{v_1,v_1v_2} =\overline{\Bp_{v_2,v_1v_2}} \in\{1,\gamma,\gamma^2\}$ in the column indexed by $v_1v_2$. 
Augment $\Bp_{x,xy} =0$ for all entries of $\Bp$ not defined so far. Then, by construction, $\Bp=\OR^\ast B \ORd$
is the incidence matrix of some mixed orientation of $\gr{T}$. Keeping in mind equation \eqref{eq:rootxdef}, we compute
\begin{align}
\begin{aligned}\label{eq:verifybpx}
 \Bp \Bp^\ast & = \OR^\ast B \ORd \ORd^\ast B^\ast \OR = \OR^\ast (A(\gr{T}) + \diag((\deg(v))_{v\in V(\gr{T})}) \OR \\
              & = \HG(\gr{X}) +  \diag((\deg(v))_{v\in V(\gr{T})}),
\end{aligned}
\end{align}
thus $\Bp$ is actually the $\gamma$\hyp{}incidence matrix of $\gr{X}$, as per \Cref{thm:identincidence}.
Moreover, using \eqref{eq:Yorient}, we conclude from
\begin{align}
\begin{aligned}\label{eq:verifybpy2}
\Bp^\ast\Bp & = (\ORd^\ast B^\ast \OR)(\OR^\ast B \ORd) = \ORd^\ast (B^\ast B) \ORd \\
            & =  \ORd^\ast A(\Gamma(\gr{Y})) \ORd + 2I =  \ORd^\ast \ORd \HG(\gr{Y}) \ORd^\ast \ORd + 2I \\
						& = \HG(\gr{Y}) + 2I
\end{aligned} 
\end{align}
that indeed $\gr{Y}=\AL(\gr{X})$, as per \Cref{thm:identincidroot}.
\end{proof}

\begin{corollary}
Let $\gr{Y}$ be a mixed graph such that $\Gamma(\gr{Y})=L(\gr{T})$ for some tree $\gr{T}$.
Then $\gr{T}$ has exactly three different mixed orientations $\gr{X}$ such that 
$\gr{Y}=\AL(\gr{X})$.
\end{corollary}

It is possible to generalize \Cref{thm:monoroot} as follows: 
\begin{corollary}\label{thm:monoroot2}
Given any $\gamma$\hyp{}monograph $\gr{Y}'$ such that $\Gamma(\gr{Y}')=L(\gr{G})$
for some undirected root graph $\gr{G}$, consider a spanning tree $\gr{T}$ of $\gr{G}$ and construct the matrices $\OR$, $\ORd$ and $\Bp$ as outlined
in the proof of \Cref{thm:monoroot}. However, instead of augmenting $\Bp$ with zeroes, redefine it as $\Bp:=\OR^\ast B' \ORd$, where
$B'$ is the incidence matrix of $\gr{G}$ (not $\gr{T}$).
If for every edge $xy\in E(\gr{G})\setminus E(\gr{T})$ we have $\OR_x\OR_y=(\ORd_{xy})^2$, then the mixed graph $\gr{Y}'$ with
$\HG(\gr{Y}')=\OR^\ast A(\gr{G}) \OR$ is a root of $\gr{Y}'$.
\end{corollary}

\begin{proof}
First one needs to verify that $\Bp$ is a proper $\gamma$\hyp{}incidence matrix.
As $\Bp$ has the same zero-nonzero pattern as $B$, it suffices to check that $\Bp_{x,xy}=\overline{\Bp_{y,xy}}$ holds for each of the extra
edges $xy\in E(\gr{G})\setminus E(\gr{T})$.
Note that the redefinition of $\Bp$ merely augments any entries not yet specified.
It follows from $\OR_x\OR_y=(\ORd_{xy})^2$ that
$\Bp_{x,xy}= \OR_x B_{x,xy} \overline{\ORd_{xy}} = (\ORd^2_{xy}\overline{\OR_y})\overline{\ORd_{xy}} = \ORd_{xy}\overline{\OR_y}=\overline{\Bp_{y,xy}}$.
We see that equations \eqref{eq:prefinalrecur} and \eqref{eq:finalrecur} are now satisfied for arbitrary paths in $\gr{G}$. Consequently, the remainder of the
proof of \Cref{thm:monoroot} can now be lifted to the entire graph $\gr{G}$.
\end{proof}

Note that, whenever the conditions stated in \Cref{thm:monoroot2} are met, it permits the construction of three valid mixed roots -- hence \Cref{thm:bip_three_roots} applies.


\bibliography{hline}

\end{document}